\documentclass[review]{elsarticle}
\usepackage{amsmath}
\usepackage{graphicx,setspace}
\usepackage{amsfonts, amssymb, amsthm}
\usepackage{mathrsfs}
\usepackage{epstopdf}
\usepackage{float}
\usepackage{color}
\usepackage{bm}
\usepackage{lineno,hyperref}
\usepackage{subfigure}
\usepackage{cases}
\usepackage{geometry}
\numberwithin{equation}{section}

\begin{document}

\newtheorem{theorem}{Theorem}[section]
\newtheorem{definition}[theorem]{Definition}
\newtheorem{lemma}[theorem]{Lemma}
\newtheorem{remark}[theorem]{Remark}
\newtheorem{proposition}[theorem]{Proposition}
\newtheorem{assumption}{Assumption}
\newtheorem*{example}{Example}
\newtheorem{corollary}[theorem]{Corollary}
\newtheorem*{notation}{Notation}
\newcommand{\supp}{\mathrm{supp}}

\numberwithin{equation}{section}

\begin{frontmatter}

\title{{\bf Stochastic Perturbations in the Fractional Nonlinear Schr\"odinger Equation: Well-posedness and Blow-up
}}
\author{\normalsize{\bf Ao Zhang$^{a,}\footnote{aozhang1993@csu.edu.cn}$,
Yanjie Zhang$^{b,*}\footnote{zhangyj2022@zzu.edu.cn}$,
and Jinqiao Duan$^{c,}\footnote{duan@gbu.edu.cn}$} \\[10pt]
\footnotesize{${}^a$School of Mathematics and Statistics, HNP-LAMA, Central South University, Changsha 410083, China.} \\[5pt]
\footnotesize{${}^b$Henan Academy of Big Data, Zhengzhou University, Zhengzhou 450052,  China.}  \\[5pt]
\footnotesize{${}^c$Department of Mathematics and Department of Physics, Great Bay University, Dongguan, Guangdong 523000, China.}
}

\begin{abstract}
  This work investigates radial solutions for nonlinear fractional Schr\"odinger equations driven by multiplicative noise. Leveraging radial deterministic and stochastic Strichartz estimates, we establish local well-posedness in the energy-subcritical regime for the stochastic fractional nonlinear Schr\"odinger equation. Global existence is subsequently demonstrated through stochastic evolution of mass and energy. In focusing supercritical settings, we derive blow-up criteria via localized virial inequality, revealing how multiplicative noise measurably suppresses blow-up formation compared to deterministic dynamics.
\end{abstract}

\begin{keyword}
 Fractional nonlinear Schr\"{o}dinger equation; Well-posedness; Blow-up; Localized virial estimate
\end{keyword}

\end{frontmatter}

\bigskip
\section{Introduction}
We analyze the stochastic fractional nonlinear Schr\"odinger equation with multiplicative noise in the energy space $H^\alpha(\mathbb{R}^n)$:
\begin{equation}
\label{orie}
\left\{
\begin{aligned}
&i du-\left[(-\Delta)^{\alpha} u+\lambda |u|^{2\sigma}u\right] dt=u \circ dW(t), \quad x \in \mathbb{R}^n, \quad t \geq 0, \\
&u(0)=u_0,
\end{aligned}
\right.
\end{equation}
where $u: \mathbb{R}^{+} \times \mathbb{R}^n \rightarrow \mathbb{C}$ is a stochastic process, $\sigma \in(0, \infty)$, and $W(t)$ denotes a real-valued cylindrical Wiener process. The parameter $\lambda \in\{ \pm 1\}$ distinguishes defocusing $(\lambda=1)$ and focusing ($\lambda=-1$) nonlinearities. The fractional Laplacian $(-\Delta)^\alpha$ with regularity exponent $\alpha \in$ $\left(\frac{1}{2}, 1\right)$ generates dispersive dynamics through its Fourier symbol $|\xi|^{2 \alpha}$. The notation $\circ$ stands for the Stratonovitch integral. The evolution problem \eqref{orie} can be seen as a canonical model for a nonlocal dispersive partial differential equation with a random potential. For physical reasons, this equation must preserve the $L^2$-norm. It follows that the product arising in the right hand side of \eqref{orie} is necessarily a Stratonovitch product and that the noise is real-valued.

To formulate the stochastic equation with precision, we construct a complete filtered probability space $\left(\Omega, \mathcal{F},\left(\mathcal{F}_t\right)_{t \geq 0}, \mathbb{P}\right)$ satisfying the usual hypotheses, endowed with a family of independent $\left(\mathcal{F}_t\right)$-adapted Brownian motions $\left\{\beta_k\right\}_{k \in \mathbb{N}^{+}}$. The noise structure is parametrized by a complete orthonormal system $\left\{e_k\right\}_{k \in \mathbb{N}^{+}}$in $L^2\left(\mathbb{R}^n, \mathbb{R}\right)$ and a Hilbert-Schmidt operator $\Phi$ on $L^{2}\left(\mathbb{R}^{n}, \mathbb{R}\right)$. The operator $\Phi$ is associated with a real-valued kernel $k \in L^2\left(\mathbb{R}^n \times \mathbb{R}^n\right)$, which is radially symmetric with respect to varaible $x$. Specially, the operator $\Phi$ is defined via the integral representation:
\begin{equation}
  \Phi h(x)=\int_{\mathbb{R}^{n}} k(x, y) h(y) d y, \quad h \in L^{2}\left(\mathbb{R}^{n}, \mathbb{R}\right),
  \end{equation}
  where the kernel $k(x,y)$ satisfies the radial symmetry condition $k(x,y)=k(|x|,y)$ for the variable $x \in \mathbb{R}^n$.
Then the process
\begin{equation*}
W(t, x, \omega):=\sum_{k=1}^{\infty} \beta_{k}(t, \omega) \Phi e_{k}(x), \quad t \geq 0, \quad x \in \mathbb{R}^{n}, \quad \omega \in \Omega,
\end{equation*}
is a Wiener process on $L^{2}\left(\mathbb{R}^{n}, \mathbb{R}\right)$ with covariance operator $\Phi \Phi^{*}$, and  the equation \eqref{orie} admits equivalent It\^o representation:
\begin{equation}
i d u-\left[(-\Delta)^{\alpha} u+\lambda |u|^{2\sigma}u \right] d t=u d W-\frac{1}{2} i u F_{\Phi} d t,
\end{equation}
where the function $F_{\Phi}$ is given by
\begin{equation}
F_{\Phi}(x)=\sum_{k=1}^{\infty}\left(\Phi e_{k}(x)\right)^{2}, \quad x \in \mathbb{R}^{n}.
\end{equation}

While the problem \eqref{orie} shares structural similarities with the extensively analyzed classical stochastic nonlinear Schr\"odinger equation (corresponding to $\alpha=1$), the development of a complete well-posedness theory and characterization of blow-up phenomena for solutions in the fractional regime $\alpha\neq 1$ remain fundamental challenges in contemporary stochastic partial differential equations analysis. This paper establishes two principal contributions to the analysis of the stochastic fractional nonlinear Schr\"odinger equation. First, for radially symmetric initial data in $H^\alpha(\mathbb{R}^n)$, we identify sufficient regularity conditions on the Hilbert-Schmidt operator that guarantee global existence of solutions. In particular, the global existence holds in the energy subcritical case, i.e. $0 \leq \sigma<\frac{2\alpha}{n-2\alpha}$ in the defocusing case ($\lambda=1$) and $0 \leq \sigma<\frac{2\alpha}{n}$ in the focusing case ($\lambda=-1$). Second, in the focusing supercritical case, i.e. $\frac{2\alpha}{n}\leq \sigma \leq\frac{2\alpha}{n-2\alpha}$ and $\lambda =-1$, we characterize general blow-up criteria, demonstrating how stochastic perturbations alter the dynamics compared to deterministic settings. Notably, our analysis reveals that multiplicative noise induces a measurable suppression effect on blow-up formation, contrasting sharply with deterministic blow-up theorems \cite{BHL16}.

The deterministic fractional nonlinear  Schr\"odinger equation appears in various fields such as nonlinear optics \cite{Longhi15}, quantum physics \cite{BM23} and water propagation \cite{Pusateri14}. Inspired by the Feynman path approach to quantum mechanics, Laskin \cite{Laskin02} used the path integral over L\'evy-like quantum mechanical paths to obtain a fractional Schr\"odinger equation.  Replacing the regular Laplacian with a fractional Laplacian in the Schr\"odinger equation might model systems with non-local interactions or anomalous diffusion. The local well-posedness of deterministic fractional nonlinear  Schr\"odinger equations in Sobolev spaces for non-radial initial data was established by Hong and Sire \cite{HS15}. The analysis crucially relies on Strichartz estimates for the unitary propagator
$e^{-i t(-\Delta)^{\alpha}}$, which manifest different forms depending on spatial symmetries: general non-radial estimates (see e.g. \cite{COX11, Di18}), improved radial estimates (see e.g. \cite{GW14, Ke12}), and weighted space-time estimates (see e.g. \cite{FW11}). For non-radial data, these Strichartz estimates have a loss of derivatives. This complicates the study of local well-posedness and results in a weaker local theory compared to the standard deterministic nonlinear Schr\"odinger equation ($\alpha=1$). The derivative loss in Strichartz estimates can be eliminated by considering radially symmetric initial data. However, such strengthened estimates impose constraints on the parameter $\alpha$, specifically requiring $\alpha\in\left[\frac{n}{2n-1}, 1\right)$.

Recent advances have focused on stochastic perturbations of the classical nonlinear Schr\"odinger equation (corresponding to $\alpha=1$ in \eqref{orie}), where both additive and multiplicative noise (in It\^o or Stratonovich form) introduce random forcing into the system. A systematic investigation of well-posedness for such stochastically perturbed nonlinear Schr\"odinger equations has been undertaken by numerous authors, establishing foundational results for these models. Bouard and Debussche developed stochastic Strichartz estimates to analyze stochastic integrals in a suitable Strichartz space. This framework enabled them to establish almost sure local existence and uniqueness of solutions in both $L^2(\mathbb{R}^n)$ and $H^1(\mathbb{R}^n)$. Subsequently, by exploiting the mass conservation law of the deterministic nonlinear Schr\"odinger equation combined with It\^o's lemma, the authors extended these results to prove almost sure global existence in $L^2(\mathbb{R}^n)$ (see \cite{BD99}). Furthermore, global well-posedness for the Cauchy problem in $H^1(\mathbb{R}^n)$ follows from a prior $H^1(\mathbb{R}^n)$-norm bound derived through energy methods (see \cite{AD03}). Within this framework, foundational contributions by Barbu, R\"ockner, and Zhang \cite{BRZ16} and Hornung \cite{Ho20} rigorously addressed analogous stochastic systems. For critical regimes, recent advancements in \cite{OO20, Zh23} provided pivotal insights into solution behaviors near critical thresholds. Current literature on stochastic fractional nonlinear Schr\"odinger equations remains limited. Notably, Brze\'zniak, Hornung, and Mille established the existence of martingale solutions for such equations on compact Riemannian manifolds with initial data in $H^{\alpha}$ (see \cite[Corollary 3.5]{BHW19}).

Motivated by these developments, we analyze the impact of multiplicative noise on the fractional nonlinear Schr\"odinger equation. Central to our analysis is the derivation of stochastic generalizations for mass and energy. We establish stochastic Strichartz estimates for the fractional Schr\"odinger propagator
$e^{-i t(-\Delta)^{\alpha}}$, which, when combined with their deterministic counterparts, enable us to prove almost sure local well-posedness of \eqref{orie} in the Sobolev space $H^\alpha(\mathbb{R}^n)$ (Theorem \ref{RLT}). In addition, we extend these results to the global regime. A key technical achievement involves deriving a prior bound on the
$H^\alpha$-norm of solutions through these generalized energy principles, which ultimately establishes global existence for \eqref{orie} (Theorem \ref{thm1}).

The question of the effect of white noise on blow-up has been widely studied also in the context of Schr\"odinger equation. The blow-up phenomena in classical nonlinear Schr\"odinger equations were first mathematically studied by Bouard and Debussche \cite{AD05} in the focusing supercritical case. Using a stochastic generalization of the variance identity and a control argument, they proved that the spatially smooth noise can cause blow-up immediately with positive probability for any given smooth initial data. See also \cite{AD03} for the additive noise case. Notably, the rescaling approach introduced by Barbu, R\"ockner and Zhang \cite{BRZ17} extended beyond well-posedness analysis, proving instrumental in characterizing blow-up dynamics for classical stochastic nonlinear Schr\"odinger equations.

We investigate the blow-up criteria for the focusing stochastic fractional nonlinear Schr\"odinger equation (corresponding to $\lambda=-1$ in \eqref{orie}). Building on the methodology of Bouard and Debussche \cite{AD03, AD05}, we develop a stochastic generalization of the virial identity via It\^o’s formula. A central challenge arises from the failure of the classical variance-virial law in fractional settings. For the standard nonlinear Schr\"odinger equation ( corresponding to $\alpha=1$), the variance-virial law governs the evolution
\begin{equation}\label{VV}
\frac{1}{2} \frac{d}{d t}\left(\int_{\mathbb{R}^n}|x|^2|u(t)|^2 d x\right)=2 \operatorname{Im}\left(\int_{\mathbb{R}^n} \bar{u}(t) x \cdot \nabla u(t) d x\right)
\end{equation}
provided the initial data satisfies $\int_{\mathbb{R}^n}|x|^2\left|u_0\right|^2 d x<+\infty$ holds. However, this argument breaks down for $\alpha \neq 1$, since identity \eqref{VV} fails in this case, as one readily checks by dimensional analysis. Inspired by the idea of Boulenger, Himmelsbach and Lenzmann \cite{BHL16}, we introduce a localized virial law by replacing the unbounded function $x$ in \eqref{VV} with a suitable cutoff function $\varphi_R$ such that $\nabla \varphi_R(x) \equiv x$ for $|x| \leqslant R$ and $\nabla \varphi_R(x) \equiv 0$ for $|x| \gg R$. This regularization enables rigorous analysis of the time evolution for the localized virial quantity, thereby establishing blow-up criteria for the stochastic fractional nonlinear Schr\"odinger equation \eqref{orie} (Theorem \ref{thm2}).

This paper is organized as follows.  In Section \ref{sec2},  we introduce some notations and state our main results.
In Section \ref{sec3}, we prove the existence of a local solution of the equation \eqref{orie} by using deterministic and stochastic fractional Strichartz inequalities. Meanwhile, we investigate the global existence of the original equation \eqref{orie}. In Section \ref{sec4}, we establish a sufficient criterion for blow-up of solution.

\section{Notations and main results}\label{sec2}
We introduce some notations throughout this paper. For $a\geq 1$, $a^{\prime}$ is called the conjugate of $a$ if it satisfies $\frac{1}{a}+\frac{1}{a^{\prime}}=1$. The capital letter $C$ denotes a positive constant, whose value may change from one line to another. The notation $C(p)$ is used to emphasize that the constant only depends on the parameter $p$, while $C(\cdot\cdot\cdot)$ is used for the case that there are
more than one parameter. For $p\geq 2$, the notation $L^{p}$ denotes the Lebesgue space of complex-valued functions.  The inner product in $L^{2}\left(\mathbb{R}^{n}\right)$ is endowed with
\begin{equation*}
  \left\langle f, g \right\rangle={\bf{Re}} \int_{\mathbb{R}^{n}} f(x) \bar{g}(x) dx,
\end{equation*}
for $f, g \in L^{2}\left(\mathbb{R}^{n}\right) $.

 Given two separable Hilbert spaces $H$ and $\widetilde{H}$. The operator $\Phi$ is called the Hilbert-Schmidt operator if there is an orthonormal basis $(e_{k})_{k\in \mathbb{N}^+}$ in $H$ such that
\begin{equation*}
\begin{aligned}
&\|\Phi\|^2_{L_{HS}(H;\widetilde{H} )}=\operatorname{tr} \Phi^* \Phi=\sum_{k \in \mathbb{N}^+}\left|\Phi e_k\right|_{\widetilde{H}}^2< \infty. \\
\end{aligned}
\end{equation*}

In the following, we review the concept of fractional derivatives. The fractional Laplace $(-\Delta)^{\alpha}$ is defined as
\begin{equation*}
(-\Delta)^{\alpha}u(x)=C(n,\alpha)~ \mathrm{P.V.}\int_{\mathbb{R}^n}\frac{u(x)-u(y)}{|x-y|^{n+2\alpha}}dy,\\
\end{equation*}
where $\mathrm{P.V.}$ denotes the principal value of the integral, and $C(n,\alpha)$ is a positive constant given by
\begin{equation}\label{constC}
\begin{aligned}
C(n,\alpha)=\frac{2^{2\alpha}\alpha\Gamma(\frac{n+2\alpha}{2})}{\pi^{\frac{n}{2}}\Gamma(1-\alpha)}.
\end{aligned}
\end{equation}
For every $\alpha\in(0,1)$, denote by
\begin{equation*}
\begin{aligned}
H^{\alpha}(\mathbb{R}^n)=\left\{u\in L^2(\mathbb{R}^n): \int_{\mathbb{R}^n}\int_{\mathbb{R}^n}\frac{|u(x)-u(y)|^2}{|x-y|^{n+2\alpha}}dxdy < \infty \right\}.
\end{aligned}
\end{equation*}
Then $H^{\alpha}(\mathbb{R}^n)$ is a Hilbert space with inner product given by
\begin{equation*}
\langle u,v \rangle_{H^{\alpha}(\mathbb{R}^n)}=\langle u,v \rangle_{L^{2}(\mathbb{R}^n)}+\mathbf{Re} \int_{\mathbb{R}^n}\int_{\mathbb{R}^n}\frac{\left(u(x)-u(y)\right)\left(\bar{v}(x)-\bar{v}(y)\right)}{|x-y|^{n+2\alpha}}dxdy,\quad u,v\in H^{\alpha}(\mathbb{R}^n).
\end{equation*}
In the following, we will also give the notation
\begin{equation*}
  \langle u,v \rangle_{\dot{H}^{\alpha}(\mathbb{R}^n)}=\mathbf{Re}\int_{\mathbb{R}^n}\int_{\mathbb{R}^n}\frac{\left(u(x)-u(y)\right)\left(\bar{v}(x)-\bar{v}(y)\right)}{|x-y|^{n+2\alpha}}dxdy,\quad u, v\in H^{\alpha}(\mathbb{R}^n),
\end{equation*}
and
\begin{equation*}
\|u\|^2_{\dot{H}^{\alpha}(\mathbb{R}^n)}=\int_{\mathbb{R}^n}\int_{\mathbb{R}^n}\frac{\left|u(x)-u(y)\right|^2}{|x-y|^{n+2\alpha}}dxdy,\quad u\in H^{\alpha}(\mathbb{R}^n).
\end{equation*}
Then we have
\begin{equation*}
  \langle u,v \rangle_{H^{\alpha}(\mathbb{R}^n)}=\langle u,v \rangle_{L^2(\mathbb{R}^n)}+\langle u,v \rangle_{\dot{H}^{\alpha}(\mathbb{R}^n)},\quad u, v\in H^{\alpha}(\mathbb{R}^n),
\end{equation*}
and
\begin{equation*}
\|u\|^2_{H^{\alpha}(\mathbb{R}^n)}=\|u\|^2_{L^2(\mathbb{R}^n)}+\|u\|^2_{\dot{H}^{\alpha}(\mathbb{R}^n)},\quad u\in H^{\alpha}(\mathbb{R}^n).
\end{equation*}
By \cite{EG12}, we know that
\begin{equation*}
\|(-\Delta)^{\frac{\alpha}{2}}u\|^{2}_{L^{2}(\mathbb{R}^n)}=\frac{1}{2}C(n,\alpha)\|u\|^{2}_{\dot{H}^{\alpha}(\mathbb{R}^n)},
\end{equation*}
and the norm $\|u\|_{H^{\alpha}\left(\mathbb{R}^n\right)}$ is equivalent to the norm $\left(\|u\|_{L^2\left(\mathbb{R}^n\right)}^2+\left\|(-\Delta)^{\frac{\alpha}{2}} u\right\|_{L^2\left(\mathbb{R}^n\right)}^2\right)^{\frac{1}{2}}$ for any $u \in H^{\alpha}\left(\mathbb{R}^n\right)$.

The fractional nonlinear Schr\"odinger equation inherits fundamental conservation laws from its classical counterpart, preserving both mass
\begin{equation}
M(u)=\int_{\mathbb{R}^n} |u|^{2}dx,
\end{equation}
and energy
\begin{equation}\label{Energy}
H(u)=\frac{1}{2}\int_{\mathbb{R}^n} \left|(-\Delta)^{\frac{\alpha}{2}}u\right|^2dx +\frac{\lambda}{2\sigma +2}\int_{\mathbb{R}^n} |u|^{2\sigma+2}dx,
\end{equation}
which underpin the analysis of solution dynamics across both deterministic and stochastic frameworks.

\begin{assumption}\label{assump1}(On the noise)
  We assume that $\Phi: L^2 \rightarrow H^{\alpha}$ is a Hilbert-Schmidt operator, i.e.
  \begin{equation*}
  \|\Phi\|_{L_{H S}(L^2,  H^{\alpha})}:=\left(\sum_{k=1}^{\infty}\left\|\Phi e_k\right\|_{ H^{\alpha}}^2\right)^{1 / 2}<\infty .
  \end{equation*}

  This implies that
  \begin{equation*}
  \|\Phi\|_{L_{H S}(L^2;  H^{\alpha})}^2=\sum_{k=1}^{\infty}\left\|\Phi e_k\right\|^2+\sum_{k=1}^{\infty}\left\|(-\Delta)^{\frac{\alpha}{2}} \Phi e_k\right\|^2<\infty.
  \end{equation*}
  \end{assumption}

  \begin{assumption}\label{assump2}(On the nonlinearity)
  \begin{enumerate}
   \item [\textcolor{black}{$\bullet$}] If $\lambda=-1$ (focusing), let $0 \leq \sigma<\frac{2\alpha}{n}$.
   \item [\textcolor{black}{$\bullet$}] If $\lambda=1$ (defocusing), let $\begin{cases}0 \leq \sigma<\frac{2\alpha}{n-2\alpha}, & \text { for } n \geq 2, \\ \sigma \geq 0, & \text { for } n < 2.\end{cases}$
   \end{enumerate}
  \end{assumption}

  Let us start with the global well-posedness in the energy-subcritical case.
  \begin{theorem}(Global well-posedness)
    \label{thm1}
    Let $n \geq 2$ and $\alpha \in\left[\frac{n}{2 n-1}, 1\right)$.  Suppose that Assumptions \ref{assump1} and \ref{assump2} hold, and
     \begin{equation}\label{Phicond}
    \sum_{k=1}^{\infty}\int_{\mathbb{R}^n}\frac{\left(\Phi e_k(x)-\Phi e_k(y)\right)^2}{|x-y|^{n+2\alpha}}dx \leq \mathcal{R}
     \end{equation}
     with a positive constant $\mathcal{R}$. Then if radial initial data $u_0\in H^{\alpha}(\mathbb{R}^n)$ almost surely, there exists a unique global solution of equation \eqref{orie} in $H^{\alpha}(\mathbb{R}^n)$.
    \end{theorem}
    \begin{remark}
      1) Due to the limitations of radial Strichartz estimates, we need to impose more constraints on the coefficient $\alpha$. As a result, compared to the classical stochastic nonlinear Schr\"odinger equation, the range of the nonlinearity exponent $\sigma$ becomes narrower.\\
      2) The condition \eqref{Phicond} is technical and reasonable, see the proof of Theorem \ref{thm1} for details.
      Here we introduce a smooth function $\rho(x)$ defined for $0 \leq x<\infty$ such that $0 \leq \rho(x) \leq 1$, for $x\geq 0$ and
      \begin{equation*}
      \label{pho}
      \rho(x)=\begin{cases}0, & \text { if }~ 0 \leq x\leq \frac{1}{2}, \\ 1, & \text { if }~ x \geq 1.\end{cases}
      \end{equation*}
      Taking $\Phi e_k(x)=\rho(\frac{|x|}{k})$ and $\alpha \in (\frac{1}{2}, 1)$, there exists a positive constant $L_1$ independent of $k$ and $y\in \mathbb{R}^n$  such that
      \begin{equation*}
      \sum_{k=1}^{\infty}\int_{\mathbb{R}^n} \frac{\left|\Phi e_k(x)-\Phi e_k(y)\right|^2}{|x-y|^{n+2\alpha}} d x \leq \sum_{k=1}^{\infty}\frac{L_1}{k^{2\alpha}}< \infty.
      \end{equation*}
      \end{remark}

Our second main result establishes a sufficient criterion for the blow-up of radial solutions in the focusing supercritical case.
\begin{theorem}[Blow-up criterion]\label{thm2}
Let $\lambda=-1$, $n\geq 2$, $\alpha \in \left[\frac{n}{2n-1}, 1\right), \frac{2\alpha}{n}\leq \sigma \leq\frac{2\alpha}{n-2\alpha}$ with $\sigma < 2\alpha$. Let $u_0$ and $\Phi$ be as in Theorem \ref{thm1}. Assume also that
for some $t>0$,
  \begin{equation}\label{Condition}
    \mathbb{E}[H(u_0)]+\frac{1}{4}tC(n,\alpha){\mathcal{R}}\mathbb{E}[M(u_0)]<0,
  \end{equation}
where $C(n,\alpha)$ is a constant defined by \eqref{constC}.
Then the solution $u$ of the equation \eqref{orie} blows up in finite time almost surely.
\end{theorem}
\begin{remark}
  From condition \eqref{Condition}, it can be seen that compared to the deterministic result \cite[Theorem 1]{BHL16}, the presence of random perturbations requires more negative initial energy to counteract the influence of the second term in the inequality \eqref{Condition}. This partially explains why spatially smooth noise can prevent blow-up phenomena.
\end{remark}

\section{Global existence}\label{sec3}
This section is devoted to investigating the global existence for the stochastic fractional nonlinear Schr\"odinger equation \eqref{orie} with radial initial data in the energy-subcritical regime. We say a pair $(p, q)$ satisfies the fractional admissible condition if
\begin{equation}
  p \in[2, \infty], \quad q \in[2, \infty), \quad(p, q) \neq\left(2, \frac{4 n-2}{2 n-3}\right), \quad \frac{2\alpha}{p}+\frac{n}{q}=\frac{n}{2}.
\end{equation}
The unitary propagator $S(t):=e^{-i t(-\Delta)^{\alpha}}$
 associated with the fractional Laplacian admits various Strichartz-type estimates. These include non-radial estimates, radial-restricted estimates, and weighted space estimates. For our current purposes, we shall specifically employ the radial Strichartz estimates (see, e.g., Ref. \cite{VDD18}).
\begin{lemma}
\label{0estimate}
(Radial Strichartz estimates) For $n \geq 2$ and $\frac{n}{2 n-1} \leq \alpha<1$, there exists a positive constant $C$ such that the following estimates hold:
  \begin{equation}\label{Strichartz}
  \begin{aligned}
  \left\|S(\cdot)\psi\right\|_{L^p\left(\mathbb{R}, L^q(\mathbb{R}^n)\right)} &\leq C\|\psi\|_{L^2(\mathbb{R}^n)},\\
  \left\|\int_{0}^{t} S(t-s) f(s) d s\right\|_{L^{p}\left(\mathbb{R}, L^{q}(\mathbb{R}^n)\right)} & \leq C \|f\|_{L^{\beta^{\prime}}\left(\mathbb{R}, L^{\gamma^{\prime}}(\mathbb{R}^n)\right)},
  \end{aligned}
  \end{equation}
  where $\psi$ and $f$ are radially symmetric, $(p, q)$ and $(\beta, \gamma)$ satisfy the fractional admissible condition, and $\frac{1}{\beta^{\prime}}+\frac{1}{\beta}=\frac{1}{\gamma^{\prime}}+\frac{1}{\gamma}=1$.
\end{lemma}

 We recall the following fractional chain rule \cite{CW91},  which is needed in the local well-posedness for \eqref{orie}.
\begin{lemma}
\label{lemm1}
Let $F \in C^1(\mathbb{C}, \mathbb{C})$ and $\alpha \in(0,1)$. Then for $1<q \leq q_2<\infty$ and $1<q_1 \leq \infty$ satisfying $\frac{1}{q}=\frac{1}{q_1}+\frac{1}{q_2}$, there exists a positive constant $C$ such that
\begin{equation}
\left\||\nabla|^{\alpha}F(u)\right\|_{L^q} \leq C\left\|F^{\prime}(u)\right\|_{L^{q_1}}\left\||\nabla|^{\alpha}u\right\|_{L^{q_2}} .
\end{equation}
\end{lemma}

Next, we establish the stochastic Strichartz estimates. Define the stochastic integral
\begin{equation*}
J u(t):=\int_{0}^{t} S(t-s)(u(s) d W(s)).
\end{equation*}
We need the following lemma, which is a straightforward generalization of Lemma 3.1 and Lemma 3.2 in \cite{BD99}.
\begin{lemma}\label{lemm2}
 Assume that a pair $(p, q)$ satisfies the fractional admissible condition, $q<\frac{2 n}{n-\alpha}$, $T_0>0$, and take $\rho \geq p$. Under Assumption \ref{assump1}, for any $u$ radially adapted process in $L^\rho\left(\Omega ; L^p\left(0, T_0 ; W^{\alpha, q}\right)\right) \cap L^\rho\left(\Omega ; L^{\infty}\left(0, T_0 ; H^{\alpha}\right)\right)$ if $Iu$ is defined for $t_0, t \in\left[0, T_0\right]$ by
\begin{equation*}
I u\left(t_0, t\right)=\int_0^{t_0} S(t-s)(u(s) d W(s)).
\end{equation*}
Then for any stopping time $\tau$ with $\tau<T_0$ almost surely,
\begin{equation*}
\mathbb{E}\left(\sup _{0 \leq t_0 \leq \tau}\left|I u\left(t_0, \cdot\right)\right|_{L^p\left(0, \tau ; W^{\alpha, q}\right)}^\rho\right) \leq C\left(\rho, T_0,\|\Phi\|_{L_{H S}(L^2,  H^{\alpha})}\right) T_0^\delta \mathbb{E}\left(|u|_{L^{\infty}\left(0, \tau ; H^{\alpha}\right)}^\rho\right)
\end{equation*}
with $\delta>0$.
Furthermore, $J u$ has a modification in $L^\rho\left(\Omega ; L^p\left(0, T_0 ; W^{\alpha, q}\right) \cap C\left(\left[0, T_0\right] ; H^{\alpha}\right)\right)$ and
\begin{equation*}
\mathbb{E}\left(|J u|_{L^p\left(0, \tau ; W^{\alpha, q}\right)}^\rho\right) \leq C\left(\rho, T_0,\|\Phi\|_{L_{H S}(L^2,  H^{\alpha})}\right) T_0^\delta \mathbb{E}\left(|u|_{L^{\infty}\left(0, \tau ; H^{\alpha}\right)}^\rho\right)
\end{equation*}
and
\begin{equation*}
\mathbb{E}\left(\sup _{0 \leq t \leq \tau}|J u(t)|_{H^{\alpha}}^\rho\right) \leq C\left(\rho, T_0,\|\Phi\|_{L_{H S}(L^2,  H^{\alpha})}\right) T_0^\delta \mathbb{E}\left(|u|_{L^p\left(0, \tau ; W^{\alpha, q}\right)}^\rho\right).
\end{equation*}
\end{lemma}

We now establish local well-posedness for the stochastic fractional nonlinear Schr\"odinger equation \eqref{orie} with radially symmetric initial data in the energy-subcritical regime.
\begin{theorem}(Radial local theory).
\label{RLT}
 Let $n \geq 2$ and $\alpha \in\left[\frac{n}{2 n-1}, 1\right)$ and $0<\sigma<\frac{2 \alpha}{n-2 \alpha}$. Let
\begin{equation*}
p=\frac{4 \alpha(\sigma+2)}{\sigma(n-2 \alpha)}, \quad q=\frac{n(\sigma+2)}{n+\sigma \alpha}. \end{equation*}
Then under Assumption \ref{assump1}, for any $u_0 \in H^{\alpha}$ radial, there exists a stopping time $\tau^*\left(u_0, \omega\right)$ and a unique solution to \eqref{orie} starting from $u_0$, which is almost surely in $C\left([0, \tau] ; H^{\alpha}\right) \cap L^p\left(0, \tau ; W^{\alpha, q}\right)$ for any $\tau<\tau^*\left(u_0\right)$. Moreover, we have almost surely,
\begin{equation*}
\tau^*\left(u_0, \omega\right)=+\infty \quad \text { or } \quad \limsup _{t\to\tau^*\left(u_0, \omega\right)}|u(t)|_{H^{\alpha}}=+\infty.
\end{equation*}
\end{theorem}

\begin{proof}
It is easy to check that $(p, q)$ satisfies the fractional admissible condition. We choose $(a, b)$ so that
\begin{equation*}
\frac{1}{p^{\prime}}=\frac{1}{p}+\frac{2\sigma}{a}, \quad \frac{1}{q^{\prime}}=\frac{1}{q}+\frac{2\sigma}{b} .
\end{equation*}
We see that
\begin{equation*}
\frac{2\sigma}{a}-\frac{2\sigma}{p}=1-\frac{\sigma(n-2 \alpha)(\sigma+1)}{2 \alpha(\sigma+2)}=: \theta>0, \quad q \leq b=\frac{n q}{n-\alpha q},\quad q<\frac{2 n}{n-\alpha}.
\end{equation*}
The latter fact gives the Sobolev embedding $W^{\alpha, q} \hookrightarrow L^b$. Define the space:
\begin{equation*}
X:=\left\{C\left(I, H^\alpha\right) \cap L^p\left(I, W^{\alpha, q}\right):\|u\|_{L^{\infty}\left(I, H^\alpha\right)}+\|u\|_{L^p\left(I, W^{\alpha, q}\right)} \leq M\right\},
\end{equation*}
where $I=[0, \zeta]$ and $M, \zeta>0$ to be chosen later. Then the proof is based on a fixed point argument. Let us briefly discuss the proof, and more details are similar to the Theorem 4.1 in \cite{AD03}. By Duhamel's formula, it suffices to prove that the functional
\begin{equation*}
\begin{aligned}
  \mathcal{T}(u)(t):= & S(t) u_0-i \lambda \int_0^t S(t-s)|u(s)|^{2 \sigma} u(s)ds \\
& +i \int_0^t S(t-s)(u(s) d W(s))-\frac{1}{2} \int_0^t S(t-s)\left(u(s) F_\Phi\right) d s
\end{aligned}
\end{equation*}
is a contraction in $L^{\rho}\left(\Omega ; X\right)$.

Let $u_1, u_2 \in L^{\rho}\left(\Omega ; X\right)$. Then using Strichartz estimates \ref{Strichartz}, we have almost surely
\begin{equation*}
\begin{aligned}
\left|\mathcal{T} u_1-\mathcal{T}u_2\right|_{X} \leq
&C\left\|\left(|u_1|^{2 \sigma} u_1-|u_2|^{2 \sigma} u_2\right)\right\|_{L^{p^{\prime}}\left(I, W^{\alpha, q^{\prime}}\right)}+\left\|\int_0^t S(t-s)\left(u_1(s)-u_2(s)\right) d W(s)\right\|_{X}\\
&\quad+C\left\|\left(u_1-u_2\right) F_\Phi\right\|_{L^{p^{\prime}}\left(I, W^{\alpha, q^{\prime}}\right)} \\
:= & H_1+H_2+H_3.
\end{aligned}
\end{equation*}
The fractional chain rule given in Lemma \ref{lemm1} and H\"older's inequality give
\begin{equation*}
\begin{aligned}
\left\||u|^{2\sigma} u\right\|_{L^{p^{\prime}}\left(I, W^{\alpha, q^{\prime}}\right)} & \leq C\|u\|_{L^a\left(I, L^b\right)}^{2\sigma}\|u\|_{L^p\left(I, W^{\alpha, q}\right)}
  \leq C|I|^\theta\|u\|_{L^p\left(I, L^b\right)}^{2\sigma}\|u\|_{L^p\left(I, W^{\alpha, q}\right)}
& \leq C|I|^\theta\|u\|_{L^p\left(I, W^{\alpha, q}\right)}^{{2\sigma}+1}.
\end{aligned}
\end{equation*}
Similarly,
\begin{equation*}
\begin{aligned}
H_1=\left\||u_1|^{2\sigma}u_1-|u_2|^{2\sigma}u_2\right\|_{L^{p^{\prime}}\left(I, W^{\alpha, q^{\prime}}\right)} & \leq C\left(\|u_1\|_{L^a\left(I, L^b\right)}^{2\sigma}+\|u_2\|_{L^a\left(I, L^b\right)}^{2\sigma}\right)\|u_1-u_2\|_{L^p\left(I, W^{\alpha, q}\right)} \\
& \leq C|I|^\theta\left(\|u_1\|_{L^p\left(I, W^{\alpha, q}\right)}^{2\sigma}+\|u_2\|_{L^p\left(I, W^{\alpha, q}\right)}^{2\sigma}\right)\|u_1-u_2\|_{L^p\left(I, W^{\alpha, q}\right)}.
\end{aligned}
\end{equation*}
The last term is easily estimated thanks to H\"older's inequality, i.e.,
\begin{equation*}
\begin{aligned}
H_3& \leq C \|I\|^{1-\frac{2}{p}}\left|u_1-u_2\right|_{L^p\left(I, W^{\alpha, q}\right)}\left|F_\Phi\right|_{L^{\frac{n p}{4\alpha}}}
 \leq C \|I\|^{1-\frac{2}{p}}\|\Phi\|_{L_{H S}(L^2,  H^{\alpha})}\left|u_1-u_2\right|_X.
\end{aligned}
\end{equation*}
Thus, using Lemma \ref{lemm2} easily shows that $\mathcal{T}$ is a contraction mapping in $L^{\rho}\left(\Omega, X\right)$ provided $\zeta$ is chosen sufficiently small, depending on $M$.
\end{proof}

We establish stochastic evolution equations governing the mass and energy dynamics for solutions to the Cauchy problem associated with the fractional equation with multiplicative noise. These results constitute a natural extension of the classical conservation laws governing the deterministic system.
\begin{proposition}\label{prop1}
Suppose that the hypotheses of Theorem \ref{RLT} still hold. For any stopping time $\tau<\tau^*(u_0)$, we have
\begin{equation}\label{mass}
M(u(\tau))=M(u_0),\quad a.s.,
\end{equation}
and
\begin{equation}\label{energy}
\begin{aligned}
H(u(\tau))&=H(u_0)+\mathbf{Im} \int^{\tau}_0 \int_{\mathbb{R}^n} (-\Delta)^{\alpha}\bar{u}udxdW
-\frac{1}{2}\mathbf{Re}\int^{\tau}_0 \int_{\mathbb{R}^n}\left[(-\Delta)^{\alpha}\bar{u}\right]uF_{\Phi}dxdt\\
&\quad+\frac{1}{2}\sum_{k=1}^{\infty} \int^{\tau}_0 \int_{\mathbb{R}^n} (-\Delta)^{\alpha}\left(\bar{u}\Phi e_k(x)\right){u}\Phi e_k(x)dxdt.
\end{aligned}
\end{equation}
\end{proposition}
\begin{proof}
The formal application of It\^o formula to the mass functional $M(u(\tau))$ and energy functional $H(u(\tau))$ yields the candidate identities \eqref{mass} and \eqref{energy}. Rigorous justification of these stochastic evolution equations is achieved through truncation procedures combined with the local well-posedness framework developed in \cite{AD03}.
\end{proof}
\begin{proof}[{\bf Proof of Theorem \ref{thm1}}]
  To establish the global existence of solutions, it suffices to demonstrate uniform boundedness of the norm $\|u\|_{H^{\alpha}}$. So we first obtain the uniform boundedness of the energy $H(u)$. The stochastic energy evolution \eqref{energy} of $u$ implies that for any $T_0>0$, any stopping time $\tau<\inf(T_0, \tau^*(u_0))$ and any time $t\leq \tau$,
\begin{equation*}
\begin{aligned}
\mathbb{E}\left[\sup_{t\leq \tau}H(u(t))\right]&\leq H(u_0)+\mathbb{E}\left[\sup_{t\leq \tau} \left|\int^{t}_{0} \int_{\mathbb{R}^n}(-\Delta)^{\alpha}\bar{u}udxdW\right|\right]\\
&\quad+\frac{1}{2}\mathbb{E}\left[\sup_{t\leq \tau} \left|\sum_{k=1}^{\infty} \int^{t}_0 \int_{\mathbb{R}^n} (-\Delta)^{\alpha}\left(\bar{u}\Phi e_k(x)\right){u}\Phi e_k(x)dxds-\mathbf{Re}\int^{t}_0 \int_{\mathbb{R}^n}\left[(-\Delta)^{\alpha}\bar{u}\right]uF_{\Phi}dxds\right|\right]\\
&=: H(u_0)+\mathcal{J}_1+\mathcal{J}_2.
\end{aligned}
\end{equation*}

For the term $\mathcal{J}_1$, one has
\begin{equation}
  \begin{aligned}
    &\mathbf{Im}\left[\int^{t}_{0} \int_{\mathbb{R}^n}(-\Delta)^{\alpha}\bar{u}udxdW\right]=\mathbf{Im}\left[\sum_{{k=1}}^{\infty} \int^{t}_{0} \int_{\mathbb{R}^n}(-\Delta)^{\frac{\alpha}{2}}\bar{u}\cdot(-\Delta)^{\frac{\alpha}{2}}(u\Phi e_k)dxd\beta_k\right]\\
  & =\frac{C(n, \alpha)}{2}\sum_{{k=1}}^{\infty}\mathbf{Im}\left[\int^{t}_{0}\int_{\mathbb{R}^n} \int_{\mathbb{R}^n} \frac{\left(\bar{u}(x)-\bar{u}(y)\right)\left(u(x)\Phi e_k(x)-u(y)\Phi e_k(y)\right)}{|x-y|^{n+2 \alpha}} dydxd\beta_k\right] \\
  &=\frac{C(n, \alpha)}{2}\sum_{{k=1}}^{\infty}\mathbf{Im}\left[\int^{t}_{0}\int_{\mathbb{R}^n} \int_{\mathbb{R}^n} \frac{\Phi e_k(x)\left|u(x)-u(y)\right|^2}{|x-y|^{n+2 \alpha}} dydxd\beta_k\right] \\
  &\quad+\frac{C(n, \alpha)}{2}\sum_{{k=1}}^{\infty}\mathbf{Im}\left[\int^{t}_{0}\int_{\mathbb{R}^n} \int_{\mathbb{R}^n} \frac{\left(\bar{u}(x)-\bar{u}(y)\right)\left(\Phi e_k(x)-\Phi e_k(y)\right)u(y)}{|x-y|^{n+2 \alpha}} dydxd\beta_k\right] \\
  &=\frac{C(n, \alpha)}{2}\sum_{{k=1}}^{\infty}\mathbf{Im}\left[\int^{t}_{0}\int_{\mathbb{R}^n} \int_{\mathbb{R}^n} \frac{\left(\bar{u}(x)-\bar{u}(y)\right)\left(\Phi e_k(x)-\Phi e_k(y)\right)u(y)}{|x-y|^{n+2 \alpha}} dydxd\beta_k\right].\\
  \end{aligned}
  \end{equation}
By using Burkholder-Davis-Gundy inequality, H\"older inequality and Young inequality, we have
\begin{equation}\label{J1}
  \begin{aligned}
  \mathcal{J}_1&=\mathbb{E}\left[\sup_{t\leq \tau} \left|\int^{t}_{0} \int_{\mathbb{R}^n}(-\Delta)^{\alpha}\bar{u}udxdW\right|\right]\\
  &\leq C\mathbb{E}\left[\left(\int^{\tau}_{0}\left(\sum_{{k=1}}^{\infty}\int_{\mathbb{R}^n} \int_{\mathbb{R}^n} \frac{\left(\bar{u}(x)-\bar{u}(y)\right)\left(\Phi e_k(x)-\Phi e_k(y)\right)u(y)}{|x-y|^{n+2 \alpha}} dxdy\right)^2\mathrm{d}t\right)^\frac{1}{2}\right] \\
  &\leq C\mathbb{E}\left[\left(\int^{\tau}_{0}\left(\sum_{{k=1}}^{\infty}\int_{\mathbb{R}^n} \left(\int_{\mathbb{R}^n} \frac{|u(y)|\left|u(x)-u(y)\right|\left|\Phi e_k(x)-\Phi e_k(y)\right|}{|x-y|^{n+2 \alpha}} dx\right) dy\right)^2dt\right)^\frac{1}{2}\right] \\
  &\leq C\mathbb{E}\left[\left(\int^{\tau}_{0}\left(\sum_{{k=1}}^{\infty}\|u\|_{L_x^2}^2\int_{\mathbb{R}^n} \left(\int_{\mathbb{R}^n} \frac{\left|\left(u(x)-u(y)\right)\left(\Phi e_k(x)-\Phi e_k(y)\right)\right|}{|x-y|^{n+2 \alpha}} dx\right)^2 dy\right)dt\right)^\frac{1}{2}\right] \\
  &\leq C(M(u_0))\mathbb{E}\left[\left(\int^{\tau}_{0}\left(\sum_{{k=1}}^{\infty}\int_{\mathbb{R}^n} \left(\int_{\mathbb{R}^n} \frac{\left|u(x)-u(y)\right|^2}{|x-y|^{n+2 \alpha}} dx \int_{\mathbb{R}^n} \frac{\left|\Phi e_k(x)-\Phi e_k(y)\right|^2}{|x-y|^{n+2 \alpha}} dx\right) dy\right)dt\right)^\frac{1}{2}\right]\\
  &\leq C(n,\alpha,T_0,M(u_0), \mathcal{R})\mathbb{E}\left[\left(\int^{\tau}_{0}\|(-\Delta)^\frac{\alpha}{2}u\|^2dt\right)^\frac{1}{2}\right]\\
  &\leq C(n,\alpha,T_0,M(u_0), \mathcal{R})+C(n,\alpha,T_0,M(u_0), \mathcal{R})\mathbb{E}\left[\left(\int^{\tau}_{0}\|(-\Delta)^\frac{\alpha}{2}u\|^4dt\right)^\frac{1}{2}\right]\\
  &\leq C(n,\alpha,T_0,M(u_0), \mathcal{R})+C(n,\alpha,T_0,M(u_0), \mathcal{R})\varepsilon\mathbb{E}\left[\sup_{t\leq \tau}\|(-\Delta)^\frac{\alpha}{2}u\|^2\right]\\
  &\quad+C(n,\alpha,T_0,M(u_0), \mathcal{R}, \varepsilon)\int^{\tau}_{0}\mathbb{E}\left(\sup_{r\leq s} \|(-\Delta)^\frac{\alpha}{2}u(r)\|^2\right)ds,
  \end{aligned}
  \end{equation}
where $\varepsilon$ is small enough.

For the term $\mathcal{J}_2$, we have
  \begin{equation}\label{energy1}
  \begin{aligned}
&\sum_{k=1}^{\infty} \int_{\mathbb{R}^n} (-\Delta)^{\alpha}\left(\bar{u}\Phi e_k(x)\right){u}\Phi e_k(x)dxdt-\mathbf{Re} \int_{\mathbb{R}^n}\left[(-\Delta)^{\alpha}\bar{u}\right]uF_{\Phi}dxdt\\
 &=\sum_{k=1}^{\infty}\mathbf{Re}\left\{\left((-\Delta)^{\frac{\alpha}{2}}\left(\bar{u}\Phi e_k\right), (-\Delta)^{\frac{\alpha}{2}}\left(\bar{u}\Phi e_k\right)\right)-\left((-\Delta)^{\frac{\alpha}{2}}\bar{u}, (-\Delta)^{\frac{\alpha}{2}}\left(\bar{u}(\Phi e_k)^2\right)\right)
 \right\}\\
 &=\frac{C(n, \alpha)}{2}\sum_{k=1}^{\infty}\int_{\mathbb{R}^n}\int_{\mathbb{R}^n}\frac{\left(\bar{u}(x)\Phi e_k(x)-\bar{u}(y)\Phi e_k(y)\right)\left(u(x)\Phi e_k(x)-u(y)\Phi e_k(y)\right)}{|x-y|^{n+2\alpha}}dxdy\\
&\quad-\frac{C(n, \alpha)}{2}\sum_{k=1}^{\infty}\mathbf{Re}\int_{\mathbb{R}^n}\int_{\mathbb{R}^n}\frac{\left(\bar{u}(x)-\bar{u}(y)\right)\left(u(x)(\Phi e_k)^2(x)-u(y)(\Phi e_k)^2(y)\right)}{|x-y|^{n+2\alpha}}dxdy\\
 &=\frac{C(n, \alpha)}{2}\sum_{k=1}^{\infty}\mathbf{Re}\int_{\mathbb{R}^n}\int_{\mathbb{R}^n}
 \frac{u(x)\bar{u}(y)\Phi e_k(x)\left(\Phi e_k(x)-\Phi e_k(y)\right)-\bar{u}(x)u(y)\Phi e_k(y)\left(\Phi e_k(x)-\Phi e_k(y)\right)}{|x-y|^{n+2\alpha}}dxdy\\
 &=\frac{C(n, \alpha)}{2}\sum_{k=1}^{\infty}\mathbf{Re}\int_{\mathbb{R}^n}\int_{\mathbb{R}^n}\frac{u(x)\bar{u}(y)\left(\Phi e_k(x)-\Phi e_k(y)\right)\left(\Phi e_k(x)-\Phi e_k(y)\right)}{|x-y|^{n+2\alpha}}dxdy\\
 &\leq\frac{C(n, \alpha)}{2}\sum_{k=1}^{\infty}\int_{\mathbb{R}^n}\int_{\mathbb{R}^n}\frac{|u(x)\left(\Phi e_k(x)-\Phi e_k(y)\right)|\cdot |u(y)\left(\Phi e_k(x)-\Phi e_k(y)\right)|}{|x-y|^{n+2\alpha}}dxdy\\
 &\leq\frac{C(n, \alpha)}{4}\sum_{k=1}^{\infty}\int_{\mathbb{R}^n}\int_{\mathbb{R}^n}\frac{|u(x)|^2\left(\Phi e_k(x)-\Phi e_k(y)\right)^2}{|x-y|^{n+2\alpha}}dxdy\\
 &\quad+\frac{C(n, \alpha)}{4}\sum_{k=0}^{\infty}\int_{\mathbb{R}^n}\int_{\mathbb{R}^n}\frac{|u(y)|^2\left(\Phi e_k(x)-\Phi e_k(y)\right)^2}{|x-y|^{n+2\alpha}}dxdy\\
 &=\frac{C(n, \alpha)}{4}\sum_{k=1}^{\infty}\left(\int_{\mathbb{R}^n}|u(x)|^2\int_{\mathbb{R}^n}\frac{\left(\Phi e_k(x)-\Phi e_k(y)\right)^2}{|x-y|^{n+2\alpha}}dy\right)dx\\
 &\quad+\frac{C(n, \alpha)}{4}
 \sum_{k=1}^{\infty}\left(\int_{\mathbb{R}^n}|u(y)|^2\int_{\mathbb{R}^n}\frac{\left(\Phi e_k(x)-\Phi e_k(y)\right)^2}{|x-y|^{n+2\alpha}}dx\right)dy\\
 &\leq \frac{C(n, \alpha)}{2}M(u_0)\mathcal{R}.
 \end{aligned}
 \end{equation}
 Thus we have
 \begin{equation}
 \label{J2}
 \mathcal{J}_2\leq C(n,\alpha,T_0,M(u_0),\mathcal{R}).
 \end{equation}
Thus combined the above inequalities \eqref{J1} and \eqref{J2}, we have
\begin{equation}
\label{1energy1}
  \begin{aligned}
    \mathbb{E}\left[\sup_{t\leq \tau}H(u(t))\right]&\leq H(u_0)+ C(n,\alpha,T_0,M(u_0), \mathcal{R})+C(n,\alpha,T_0,M(u_0), \mathcal{R})\varepsilon\mathbb{E}\left[\sup_{t\leq \tau}\|(-\Delta)^\frac{\alpha}{2}u\|^2\right]\\
  &\quad+C(n,\alpha,T_0,M(u_0), \mathcal{R}, \varepsilon)\int^{\tau}_{0}\mathbb{E}\left(\sup_{r\leq s} \|(-\Delta)^\frac{\alpha}{2}u(r)\|^2\right)ds.
  \end{aligned}
\end{equation}

We first consider the case where $\lambda=1$ and $\sigma<\frac{2\alpha}{n-2\alpha}$, then using the equality \eqref{Energy}, we obtain
\begin{equation}
\begin{aligned}
\mathbb{E}\sup_{t\leq \tau}\|u\|^2_{\dot{H}^{\alpha}}&\leq 2\mathbb{E}\sup_{t\leq \tau} H(u)\\
&\leq 2H(u_0)+ C(n,\alpha,T_0,M(u_0), \mathcal{R})+C(n,\alpha,T_0,M(u_0), \mathcal{R})\varepsilon\mathbb{E}\left[\sup_{t\leq \tau}\|(-\Delta)^\frac{\alpha}{2}u\|^2\right]\\
  &\quad+C(n,\alpha,T_0,M(u_0), \mathcal{R}, \varepsilon)\int^{\tau}_{0}\mathbb{E}\left(\sup_{r\leq s} \|(-\Delta)^\frac{\alpha}{2}u(r)\|^2\right)ds.
\end{aligned}
\end{equation}
Then, using the Gr\"onwall inequality and the conservation of mass, one can deduce
\begin{equation*}
\mathbb{E}\sup_{t\leq \tau}\|u\|^2_{H^{\alpha}}\leq C(n,\alpha,T_0,M(u_0), H(u_0),\mathcal{R}).
\end{equation*}

Treating the case where $\lambda=-1$ and $\sigma<\frac{2\alpha}{n}$, we can utilize Gagliardo--Nirenberg's inequality (see \cite{VDD19}) and Young's inequality to derive the following equation:
\begin{equation}
\label{in6}
\begin{aligned}
\|u\|^{2\sigma+2}_{L^{2\sigma+2}}&\leq C \|u\|^{\frac{n\sigma}{\alpha}}_{\dot{H}^{\alpha}}\cdot \|u\|^{2\sigma+2-\frac{n\sigma}{\alpha}}_{L^2}
\leq C \varepsilon \|u\|^{2}_{\dot{H}^{\alpha}}+ C_{\varepsilon}\|u\|^{2+\frac{4\sigma}{2\alpha-n\sigma}}_{L^2}.
\end{aligned}
\end{equation}
It's important to note that in the last inequality, it is crucial that $\sigma<\frac{2\alpha}{n}$. Then by the equality \eqref{Energy} and inequality \eqref{in6}, we know
\begin{equation*}
\begin{aligned}
\mathbb{E}\sup_{t\leq \tau}\|u\|^2_{\dot{H}^{\alpha}}&=2\mathbb{E}[\sup_{t\leq \tau} H(u)]+\frac{1}{\sigma+1}\mathbb{E}\left[\sup_{t\leq \tau}\|u\|^{2\sigma+2}_{L^{2\sigma+2}}\right]\\
&\leq 2H(u_0)+ C(n,\alpha,T_0,M(u_0), \mathcal{R})+C(n,\alpha,T_0,M(u_0), \mathcal{R})\varepsilon\mathbb{E}\left[\sup_{t\leq \tau}\|(-\Delta)^\frac{\alpha}{2}u\|^2\right]\\
  &\quad+C(n,\alpha,T_0,M(u_0), \mathcal{R}, \varepsilon)\int^{\tau}_{0}\mathbb{E}\left(\sup_{r\leq s} \|(-\Delta)^\frac{\alpha}{2}u(r)\|^2\right)ds.
\end{aligned}
\end{equation*}
By using Gr\"onwall inequality, we obtain
\begin{equation*}
\mathbb{E}\sup_{t\leq \tau}\|u\|^2_{\dot{H}^{\alpha}} \leq C( n,\alpha,T_0, M(u_0), H(u_0),\mathcal{R}).
\end{equation*}
These prior estimates, combined with local well-posedness, imply the global existence of a unique solution.
\end{proof}


\section{Blow-up criterion}\label{sec4}
This section establishes finite-time blow-up phenomena for the stochastic fractional nonlinear Schr\"odinger equation \eqref{orie} in the focusing supercritical case. We consider the following stochastic fractional nonlinear Schr\"odinger equation with focusing ($\lambda=-1$) power-type nonlinearity
\begin{equation}
  \label{orie1}
  \left\{
  \begin{aligned}
  &i du-\left[(-\Delta)^{\alpha} u- |u|^{2\sigma}u\right] dt=u \circ dW(t), \quad x \in \mathbb{R}^n, \quad t \geq 0, \\
  &u(0)=u_0,
  \end{aligned}
  \right.
  \end{equation}

  Let us assume that $\varphi: \mathbb{R}^{n} \rightarrow \mathbb{R}$ is a real-valued function with $\nabla \varphi \in W^{3, \infty}\left(\mathbb{R}^{n}\right)$. We define the localized virial of $u=u(t, x)$ to be the quantity given by
\begin{equation}
\mathcal{M}_{\varphi}[u(t)]:=2 \operatorname{Im} \int_{\mathbb{R}^{n}} \bar{u}(t) \nabla \varphi \cdot \nabla u(t) d x.
\end{equation}
Define the following self-adjoint differential operator
\begin{equation*}
\Gamma_{\varphi}:=-i\left(\nabla \cdot \nabla\varphi +\nabla\varphi\cdot \nabla\right),
\end{equation*}
which acts on functions according to
\begin{equation*}
\Gamma_{\varphi}f=-i\left(\nabla \cdot \left((\nabla\varphi) f\right) +(\nabla\varphi)\cdot (\nabla f)\right).
\end{equation*}
Then we readily check that
\begin{equation*}
  \mathcal{M}_{\varphi}[u(t)]=\left\langle u(t), \Gamma_{\varphi} u(t)\right\rangle.
  \end{equation*}
Therefore, we establish the time evolution of the localized virial $\mathcal{M}_{\varphi}[u(t)]$.
  \begin{lemma}
  Let $\alpha\in(1/2,1)$ and $\varphi: \mathbb{R}^n\rightarrow\mathbb{R}$ be such that $\nabla \varphi \in W^{3, \infty}\left(\mathbb{R}^{n}\right)$. Assume that $u\in C\left([0, \tau^*) ; H^{\alpha}\right)$ is a solution of \eqref{orie1}. Then for any $t \in[0, \tau^*)$, we have the identity
  \begin{equation}\label{LVI}
  \begin{aligned}
  d\mathcal{M}_{\varphi}[u(t)]&=\langle u, [(-\Delta)^{\alpha}, i \Gamma_{\varphi}]u \rangle dt+ \langle
  u, [-|u|^{2\sigma}, i\Gamma_{\varphi} ]u\rangle dt -2 \sum_{k=1}^{\infty} \int_{\mathbb{R}^n}|u|^2 \nabla \varphi \cdot \nabla(\Phi e_k)dx d\beta_k.
  \end{aligned}
  \end{equation}
  \end{lemma}
\begin{proof}
To simplify the presentation, we omit some procedures like mollifying the unbounded operator and taking the limit on the regularization parameter. Using It\^o's formula directly, we obtain
\begin{equation*}
\begin{aligned}
d\mathcal{M}_{\varphi}(u)&=\left\langle du, \Gamma_{\varphi} u \right\rangle
+\langle u, \Gamma_{\varphi} (du) \rangle+ \langle -iu dW, \Gamma_{\varphi}(-iudW)\rangle \\
&=\left\langle -i \left[\left((-\Delta)^{\alpha} u- |u|^{2\sigma}u\right)dt+udW-\frac{i}{2}uF_{\Phi}dt\right], \Gamma_{\varphi}u \right\rangle \\
&\quad+\left\langle u, -i\Gamma_{\varphi}\left[((-\Delta)^{\alpha}u-|u|^{2\sigma}u)dt+udW- \frac{i}{2}uF_{\Phi}dt\right]\right\rangle
+\langle udW, \Gamma_{\varphi}(udW)\rangle \\
\end{aligned}
\end{equation*}
\begin{equation*}
  \begin{aligned}
&=\left[\langle (-\Delta)^{\alpha}u, i\Gamma_{\varphi} u\rangle +\langle u, -i\Gamma_{\varphi}((-\Delta)^{\alpha} u)\rangle \right]dt \\
&+\left[\langle -|u|^{2\sigma}u, i\Gamma_{\varphi} u\rangle  + \langle u, -i\Gamma_{\varphi}(-|u|^{2\sigma}u)\rangle\right] dt \\
&+\left[\langle -iudW, \Gamma_{\varphi} u\rangle +\langle u, -i\Gamma_{\varphi}(udW) \rangle \right]\\
&+\left[\langle -\frac{1}{2}uF_{\Phi}, \Gamma_{\varphi}u \rangle  +\langle u, -\frac{1}{2}\Gamma_{\varphi}(uF_{\Phi})\rangle \right]dt\\
&+\langle udW, \Gamma_{\varphi}(udW) \rangle \\
&=:I_1+I_2+I_3+I_4+I_5.
\end{aligned}
\end{equation*}
For the terms $I_1$ and $I_2$, we have
\begin{equation*}
I_1+I_2=\left\langle u, [(-\Delta)^{\alpha}, i\Gamma_{\varphi}]u \rangle dt+\langle u, [-|u|^{2\sigma}, i\Gamma_{\varphi}]u\right \rangle dt,
\end{equation*}
where we recall that $[X, Y] \equiv X Y-Y X$ denotes the commutator of $X$ and $Y$.

For the terms $I_4$ and $I_5$, we have
\begin{equation*}
\begin{aligned}
I_4+I_5&=\langle -\frac{1}{2}uF_{\Phi}, \Gamma_{\varphi}u\rangle dt+\langle u, -\frac{1}{2}\Gamma_{\varphi}(uF_{\Phi}) \rangle dt + \sum_{k=1}^{\infty}\langle u\Phi e_k, \Gamma_{\varphi}(u\Phi e_k) \rangle dt \\
&=-\frac{1}{2}\int_{\mathbb{R}^n} uF_{\Phi}\left[\left(i\left( \nabla  \cdot (\nabla \varphi \bar{u})+\nabla \varphi \cdot \nabla \bar{u}\right)\right)\right]dx
-\frac{1}{2}\int_{\mathbb{R}^n} u\left[\left(i\left( \nabla  \cdot (\nabla \varphi \bar{u}F_{\Phi})+\nabla \varphi \cdot \nabla ( \bar{u}F_{\Phi}\right)\right)\right]dx\\
&+\sum_{k=1}^{\infty}\int_{\mathbb{R}^n}u\Phi e_k \left[i\left( \nabla  \cdot (\nabla \varphi \bar{u}\Phi e_k)+\nabla \varphi \cdot \nabla (\bar{u}\Phi e_k)\right)\right]dx\\
&=-\frac{i}{2}\sum_{k=1}^{\infty}\int_{\mathbb{R}^n} u(\Phi e_k)^2\left(\triangle \varphi \bar{u}+2 \nabla\varphi \cdot \nabla \bar{u} \right)dx\\
&-\frac{i}{2}\sum_{k=1}^{\infty}\int_{\mathbb{R}^n}u\left(\triangle \varphi \bar{u} (\Phi e_k)^2
+ 2 \nabla \varphi \cdot \nabla \bar{u}(\Phi e_k)^2+4\nabla \varphi \cdot \nabla(\Phi e_k)\bar{u}\Phi e_k\right)dx\\
&+i\sum_{k=1}^{\infty}\int_{\mathbb{R}^n} u(\Phi e_k)\left(\triangle \varphi \bar{u} \Phi e_k+ 2 \nabla \varphi \cdot \nabla(\Phi e_k ) \bar{u}+2 \nabla \varphi \cdot \nabla\bar{u}(\Phi e_k ) \right)dx\\
&=0.
\end{aligned}
\end{equation*}
For the term $I_3$, we have
\begin{equation*}
\begin{aligned}
I_3&=\langle -iudW, \Gamma_{\varphi}u \rangle +\langle u, -i \Gamma_{\varphi}(udW)\rangle \\
&=\sum_{k=1}^{\infty} \int_{\mathbb{R}^n} -iu \Phi e_k \left[i (\nabla \cdot (\nabla \varphi \bar{u}))+\nabla \varphi \cdot \nabla \bar{u}\right]dx d\beta_k\\
&+\sum_{k=1}^{\infty}\int_{\mathbb{R}^n} u\left[-(\nabla \cdot(\nabla \varphi \bar{u} \Phi e_k)+\nabla \varphi \cdot \nabla(\bar{u}\Phi e_k))\right]dx d\beta_k\\
&=\sum_{k=1}^{\infty}\int_{\mathbb{R}^n} \left(u\Phi e_k (\triangle \varphi \bar{u}+2\nabla \varphi \cdot \nabla \bar{u})-u (\triangle \varphi \bar{u}\Phi e_k +2 \nabla \varphi\cdot \nabla(\bar{u}\Phi e_k))\right)dx d\beta_k\\
&=-2\sum_{k=1}^{\infty}\int_{\mathbb{R}^n} |u|^2 \nabla \varphi \cdot \nabla(\Phi e_k)dxd\beta_k.
\end{aligned}
\end{equation*}
By summing the terms $I_1$ through $I_5$, we obtain the equality \eqref{LVI}.
\end{proof}

We have the following virial idetity.
\begin{lemma}[Localized Virial Identity]\label{LemLVI}
Let $\alpha\in(1/2,1)$ and $\varphi: \mathbb{R}^n\rightarrow\mathbb{R}$ be such that $\nabla \varphi \in W^{3, \infty}\left(\mathbb{R}^{n}\right)$. Assume that $u\in C\left([0, \tau^*) ; H^{\alpha}\right)$ is a solution of \eqref{orie1}. Then for any $t \in[0, \tau^*)$, we have the identity
\begin{equation}\label{VI}
\begin{aligned}
\frac{d}{dt}\mathbb{E}\left[\mathcal{M}_{\varphi}[u(t)]\right]&=\mathbb{E}\left[\int^{\infty}_0 m^{\alpha}\int_{\mathbb{R}^n}
\left\{4 \overline{\partial_k u_m}(\partial^2_{kl}\varphi)\partial_l u_m-(\triangle^2 \varphi)|u_m|^2\right\}dxdm\right]\\ &-\frac{2\sigma}{\sigma+1}\mathbb{E}\left[\int_{\mathbb{R}^n}(\Delta \varphi)|u|^{2\sigma+2}dx\right],
\end{aligned}
\end{equation}
where $u_m=u_{m}(t,x)$ is defined by
\begin{equation*}
u_m(t):=\frac{\sin \pi \alpha}{\pi}\frac{1}{-\triangle +m}u(t)=\frac{\sin \pi \alpha}{\pi}
\mathcal{F}^{-1}\left(\frac{\hat{u}(t, \xi)}{|\xi|^2+m}\right).
\end{equation*}
\end{lemma}
\begin{proof}
  We immediately observe that the stochastic integral $2 \sum_{k=1}^{\infty}\int_0^t \int_{\mathbb{R}^n}|u|^2 \nabla \varphi \cdot \nabla(\Phi e_k)dx d\beta_k$ is square integrable and its expectation vanishes. Consequently, applying the expectation operator to both sides of identity \eqref{LVI} eliminates the martingale terms, yielding the simplified deterministic relation. Then using the result in \cite[Lemma 2.1]{BHL16}, we can derive the equality \eqref{VI}.
\end{proof}

We now use the  formula for $\mathcal{M}_{\varphi}[u(t)]$ when $\varphi(x)$ is a suitable approximation of the unbounded function $a(x)=\frac{1}{2}|x|^2$ and hence $\nabla a(x)=x$. Let
$\varphi: \mathbb{R}^{n}\rightarrow \mathbb{R}$ be as above. We assume that $\varphi=\varphi(r)$ is radial and satisfies
\begin{equation*}
\varphi(r)=
\begin{cases}
r^2/2,  & r \leq 1, \\
constant,  & r \geq 10,
\end{cases}
\end{equation*}
and
\begin{equation*}
\varphi^{''}(r)\leq 1, ~~ r \geq 0.
\end{equation*}
For $R>0$ given, we define the rescaled function $\varphi_{R}: \mathbb{R}^n\rightarrow \mathbb{R}$ by setting
\begin{equation*}
\varphi_{R}(r):=R^2\varphi\left(\frac{r}{R}\right).
\end{equation*}
Then  we obtain the following  inequalities by a simple calculation
\begin{equation}
\label{inuq}
1-\varphi^{''}_{R}(r)\geq 0, ~~1-\frac{\varphi^{'}_{R}(r)}{r}\geq 0, ~~n-\triangle\varphi_{R}(r)\geq 0, ~~ r\geq 0,
\end{equation}

\begin{equation*}
\nabla \varphi_{R}(r)=R \varphi^{'}(\frac{r}{R})\frac{x}{|x|}=
\begin{cases}
x,  & r \leq R, \\
0,  & r \geq 10R,
\end{cases}
\end{equation*}
and
\begin{equation*}
\|\nabla^{j}\varphi_{R}\|_{L^{\infty}}\leq CR^{2-j},  ~~ 0\leq j \leq 4,
\end{equation*}

\begin{equation*}
\supp(\nabla^{j}\varphi_{R})\subset
\begin{cases}
\{|x|\leq 10R\},  & j=1,2, \\
\{R\leq |x|\leq 10R\},  & 3\leq j \leq 4.
\end{cases}
\end{equation*}

Here we give a priori estimate of $\mathbb{E}[H(u)]$.
\begin{proposition}
Suppose that the hypotheses of Theorem \ref{RLT} still hold. For any $t \in[0, \tau^*)$, we have
\begin{equation}
\label{0energy1}
\begin{aligned}
  \mathbb{E}[H(u)]&\leq \mathbb{E}[H(u_0)]+\frac{C(n,\alpha)}{4}t\mathbb{E}M(u_0)\mathcal{R},
\end{aligned}
\end{equation}
where $C(n,\alpha)$ is a constant defined by \eqref{constC}
\end{proposition}
\begin{proof}
  In the proof of Proposition \ref{prop1}, we observe that the stochastic integral in \eqref{energy} is square integrable and its expectation vanishes. For the stochastic evolution equation of energy $H(u)$, after taking expectation and using the result \eqref{energy1}, we obtain the inequality \eqref{0energy1}.
\end{proof}
For the time evolution of the localized virial $\mathcal{M}_{\varphi_{R}}[u(t)]$ with $\varphi_{R}$ as above, we have the following localized radial virial estimate.

\begin{lemma}
\label{technique}
Let $\lambda=-1$. Suppose that the hypotheses of Theorem \ref{RLT} still hold. Then for any $t \in[0, \tau^*)$, we have
\begin{equation}
  \begin{aligned}
  \frac{d}{dt}\mathbb{E}\left[\mathcal{M}_{\varphi_{R}}[u(t)]\right]&\leq 4\sigma n\mathbb{E}[H(u_0)]+\sigma nC({n,\alpha})t{\mathcal{R}}\mathbb{E}[M(u_0)]
  -2(\sigma n-2\alpha)\mathbb{E}\left[\|(-\Delta)^{\frac{\alpha}{2}}u(t)\|^2_{L^2}\right]\\
  &+C\left\{R^{-2\alpha}+CR^{-\sigma(n-1)+\varepsilon \alpha}\mathbb{E}\left[\|(-\Delta)^{\alpha/2}u(t)\|^{(\sigma/\alpha)+\varepsilon}_{L^2}\right]\right\}
  \end{aligned}
  \end{equation}
for any $0< \varepsilon < (2\alpha-1)\sigma/\alpha$. Here $C(n,\alpha)$ is a constant defined by \eqref{constC}, $C=C(\|u_0\|_{L^2}, n, \varepsilon, \alpha, \sigma)>0$ is
some constant that only depends on $\|u_0\|_{L^2}, n, \varepsilon, \alpha$ and $\sigma$.
\end{lemma}
\begin{proof}
The Hessian of a radial function $f$ can be written as
\begin{equation*}
\partial^2_{kl}f(|x|)=\left(\delta_{kl}-\frac{x_lx_k}{r^2}\right)\frac{\partial_r f}{r}+\frac{x_kx_l}{r^2}\partial^2_{r}f.
\end{equation*}
By Lemma \ref{LemLVI}, Plancherel's theorem, Fubini's theorem and inequality \eqref{inuq},  we have
\begin{equation}
\label{estimate1}
\begin{aligned}
&\int^{\infty}_{0} m^\alpha \int_{\mathbb{R}^{n}} \overline{\partial_k u_m}(\partial^2_{kl}\varphi_{R})\partial_{l}u_mdxdm
=\int^{\infty}_{0}m^\alpha\int_{\mathbb{R}^{n}} (\partial^2_{r}\varphi_{R})|\partial_{r}u_m|^2dx dm\\
&=\int_{\mathbb{R}^{n}}\left(\frac{\sin \pi \alpha}{\pi}\int^{\infty}_{0}\frac{m^{\alpha}}{(|\xi|^2+m)^2}\right)|\xi|^2 |\widehat{u(\xi)}|d\xi-
\int^{\infty}_{0}m^\alpha\int_{\mathbb{R}^{n}}\left(1-\partial^{2}_{r}\varphi_{R}\right)|\partial_{r}u_{m}|^2dxdm\\
&=\alpha\|\left(-\Delta\right)^{\frac{\alpha}{2}}u\|^{2}_{L^2}-
\int^{\infty}_{0}m^\alpha\int_{\mathbb{R}^{n}}\left(1-\partial^{2}_{r}\varphi_{R}\right)|\partial_{r}u_{m}|^2dxdm\\
&\leq \alpha\|\left(-\Delta\right)^{\frac{\alpha}{2}}u\|^{2}_{L^2}.
\end{aligned}
\end{equation}
Moreover, by \cite[Lemma A.2]{BHL16}, we obtain
\begin{equation}
\left|\int^{\infty}_{0}m^{\alpha}\int_{\mathbb{R}^n}(\triangle^2 \varphi_{R})|u_m|^2dxdm\right|
\leq C \|\triangle^2 \varphi_{R}\|^{\alpha}_{L^{\infty}}\|\triangle \varphi_{R}\|^{1-\alpha}_{L^{\infty}}\|u\|^2_{L^2}\leq C R^{-2\alpha}.
\end{equation}
Note that $\triangle \varphi_{R}(r)-n\equiv 0$ on $\{r\leq R\}$. Thus, we obtain that
\begin{equation}
\begin{aligned}
-\frac{2\sigma}{\sigma +1}\int_{\mathbb{R}^n} (\triangle \varphi_{R})|u|^{2\sigma+2}dx =
-\frac{2\sigma n}{\sigma +1}\int_{\mathbb{R}^n} |u|^{2\sigma +2}dx
-\frac{2\sigma}{\sigma +1}\int_{|x|\geq R}(\triangle \varphi_{R}-n)|u|^{2\sigma+2}dx,
\end{aligned}
\end{equation}
and
\begin{equation}
\begin{aligned}
-\frac{2\sigma}{\sigma +1}\int_{|x|\geq R}(\triangle \varphi_{R}-n)|u|^{2\sigma+2}dx &\leq \frac{2\sigma}{\sigma +1}
\|\triangle \varphi_{R}\|_{L^{\infty}}\int_{|x|\geq R}|u|^{2\sigma+2}dx+\frac{2n\sigma}{\sigma +1}\int_{|x|\geq R}|u|^{2\sigma+2}dx\\
&\leq C(n, \sigma)\int_{|x|\geq R}|u|^{2\sigma+2}dx.
\end{aligned}
\end{equation}
Let $0< \varepsilon < (2\alpha-1)\sigma/\alpha$. By the interpolation inequality, for any $\frac{1}{2} < s:=\frac{1}{2}+\frac{\varepsilon \alpha}{2\sigma} < \alpha <\frac{n}{2}$,  we have
\begin{equation}
\label{eq30}
\|(-\Delta)^{s/2}u\|_{L^2}\leq
\|u\|^{1-s/\alpha}_{L^2}\|(-\triangle)^{\alpha/2}u\|^{s/\alpha}_{L^2} \leq C\|(-\Delta)^{\alpha/2} u\|^{s/\alpha}_{L^2}.
\end{equation}
Next, we recall from \cite{CO09} the fractional radial Sobolev (generalized Strauss) inequality
\begin{equation*}
\sup _{x \neq 0}|x|^{\frac{n}{2}-s}|u(x)| \leqslant C\left\|(-\Delta)^{s / 2} u\right\|_{L^2}
\end{equation*}
for all radial functions $u \in \dot{H}^s\left(\mathbb{R}^n\right)$ provided that $1 / 2<s<n / 2$.
Thus by the generalized Strauss inequality and inequality \eqref{eq30}, we obtain
\begin{equation}
\label{estimate2}
\begin{aligned}
\int_{|x|\geq R}|u|^{2\sigma+2}dx &\leq \|u\|^{2}_{L^2}\|u\|^{2\sigma}_{L^{\infty}(|x|\geq R)} \leq C(n, \alpha, \varepsilon)R^{-2\sigma(\frac{n}{2}-s)}\|(-\Delta)^{s/2}u\|^{2\sigma}_{L^2}\\
& \leq C(n, \alpha, \varepsilon)R^{-2 \sigma (\frac{n}{2}-s)}\|(-\Delta)^{\frac{s}{2}}u\|^{2\sigma s/\alpha}_{L^2}\\
&=C(n, \alpha, \varepsilon)R^{-\sigma(n-1)+\varepsilon \alpha}\|(-\Delta)^{\alpha/2}u\|^{(\sigma/\alpha)+\varepsilon}_{L^2}.
\end{aligned}
\end{equation}

Combining \eqref{estimate1}-\eqref{estimate2}, and using inequality \eqref{0energy1},  we have the following estimate which yields
\begin{equation*}
\begin{aligned}
&\frac{d}{dt}\mathbb{E}\left[\mathcal{M}_{\varphi_{R}}[u(t)]\right]\\
&=\mathbb{E}\left[\int^{\infty}_0 m^{\alpha}\int_{\mathbb{R}^n}
\left\{4 \overline{\partial_k u_m}(\partial^2_{kl}\varphi_{R})\partial_l u_m-(\triangle^2 \varphi_{R})|u_m|^2\right\}dxdm\right]-\frac{2\sigma}{\sigma+1}\mathbb{E}\left[\int_{\mathbb{R}^n}(\Delta \varphi_{R})|u|^{2\sigma+2}dx\right]\\
& \leq 4\sigma n\mathbb{E}[H(u)]
-2(\sigma n-2\alpha)\mathbb{E}\left[\|(-\Delta)^{\frac{\alpha}{2}}u(t)\|^2_{L^2}\right]\\
&+ C\left\{R^{-2\alpha}+CR^{-\sigma(n-1)+\varepsilon \alpha}\mathbb{E}\left[\|(-\Delta)^{\alpha/2}u(t)\|^{(\sigma/\alpha)+\varepsilon}_{L^2}\right]\right\}\\
& \leq 4\sigma n\mathbb{E}[H(u_0)]+\sigma nC(n,\alpha)t{\mathcal{R}}\mathbb{E}[M(u_0)]\\
&-2(\sigma n-2\alpha)\mathbb{E}\left[\|(-\Delta)^{\frac{\alpha}{2}}u(t)\|^2_{L^2}\right]+ C\left\{R^{-2\alpha}+CR^{-\sigma(n-1)+\varepsilon \alpha}\mathbb{E}\left[\|(-\Delta)^{\alpha/2}u(t)\|^{(\sigma/\alpha)+\varepsilon}_{L^2}\right]\right\},
\end{aligned}
\end{equation*}
for $0< \varepsilon < (2\alpha-1)\sigma/\alpha$.
\end{proof}
In the following, we will present the estimate, which comes from \cite[Lemma A.1]{BHL16}.
\begin{lemma}\label{lem10}
Let $n\geq 1$ and suppose $\varphi: \mathbb{R}^{n}\rightarrow \mathbb{R}$ is such that
$\nabla \varphi \in W^{1, \infty}(\mathbb{R}^{n})$. Then for all $u \in H^{1/2}(\mathbb{R}^{n})$, it holds that
\begin{equation*}
\left|\int_{\mathbb{R}^n} \bar{u}(x)\nabla \varphi(x)\cdot \nabla u(x)dx\right|\leq C \left(\||\nabla|^{1/2}u\|^2_{L^2}+\|u\|_{L^2}\||\nabla|^{1/2}u\|_{L^2}\right),
\end{equation*}
where $C$ is a positive constant depending only on $\|\nabla \varphi\|_{W^{1, \infty}}$ and $n$.
\end{lemma}

Now we will present the lower bound of $u(t)$.
\begin{lemma}\label{lem11}
Assume that $\mathbb{E}\left[H(u(t))\right]< 0$, then for all $t\in [0,\tau^*)$, there exists a positive constant $C$ such that
\begin{equation}\label{bb1}
\mathbb{E}\left[\|(-\Delta)^{\alpha/2}u(t)\|_{{L^2}}\right]\geq C.
\end{equation}
\end{lemma}
\begin{proof}
Suppose this bound is not true. Thus for some sequence of times $t_k \in [0, \tau^*)$, we have
\begin{equation*}
\mathbb{E}\left[\|(-\Delta)^{\alpha/2}u(t_{k})\|_{{L^2}}\right]\rightarrow 0 .
\end{equation*}
By the $L^2$-mass conservation and the Gagliardo-Nirenberg inequality, we have
\begin{equation*}
\mathbb{E}\left[\|u(t_k)\|_{L^{2\sigma+2}}\right]\rightarrow 0.
\end{equation*}
However, by the definition of energy, we know $\mathbb{E}\left[H(u(t_k))\right]\rightarrow 0$. Therefore it is a contradiction to $\mathbb{E}\left[H(u(t_k))\right]< 0$. This implies the inequality \eqref{bb1} holds.
\end{proof}

Now we are in the position to finish the proof of Theorem \ref{thm2}.

\begin{proof}[\textbf{Proof of Theorem \ref{thm2}}]
Step 1. Note that $\sigma n-2\alpha\geq 0$. We deduce the inequality (with $o_{R}(1) \rightarrow 0$ as $R \rightarrow+\infty$ uniformly in $t$ ):
\begin{equation}\label{viril}
\begin{aligned}
\frac{d}{dt}\mathbb{E}\left[\mathcal{M}_{\varphi_R}[u(t)]\right] &\leq 4\sigma n  \mathbb{E}[[H(u_0)]]+\sigma nC(n,\alpha)t{\mathcal{R}}\mathbb{E}[M(u_0)]\\
&-2(\sigma n-2\alpha)\mathbb{E}\left[{\|(-\Delta)^{\alpha/2}u(t)\|}^{2}_{{L^2}}\right]
+o_{R}(1)\cdot \left(1+\mathbb{E}\left[{\|(-\Delta)^{\alpha/2}u(t)\|}^{(\sigma/\alpha)+\varepsilon}_{L^2}\right]\right)\\
& \leq 4\sigma n \mathbb{E}[H(u_0)]+\sigma nC_{n,\alpha}t{\mathcal{R}}\mathbb{E}[M(u_0)]\\
&-(\sigma n-2\alpha)\mathbb{E}\left[\|(-\Delta)^{\alpha/2}u(t)\|^2_{L^2}\right],
\end{aligned}
\end{equation}
provided that $R \gg 1$ is taken sufficiently large. In the last step, we used Young's inequality, and that $\sigma / \alpha+\varepsilon<2$ when $\varepsilon>0$ is sufficiently small. So, the condition $\sigma<2\alpha$ is needed.

Step 2. Suppose $u(t)$ exists for all times $t\geq 0$, i.e., we can take $T=\infty$. Form \eqref{Condition} and \eqref{viril}, we get
\begin{equation}
\label{es2}
\frac{d}{dt}\mathbb{E}\left[\mathcal{M}_{\varphi_R}[u(t)]\right] \leq -c
\end{equation}
with some constant $c>0$.
If we integrate  \eqref{es2} on $[t_1, t]$, we obtain
\begin{equation}
\label{M}
\begin{aligned}
\mathbb{E}\left[\mathcal{M}_{\varphi_R}[u(t)]\right]\leq -(\sigma n-2\alpha)\int^{t}_{t_1}\mathbb{E}\left[\|(-\Delta)^{\alpha/2}u(\tau)\|^2_{L^2}\right]d \tau
 \leq 0.
\end{aligned}
\end{equation}
As $\alpha>1/2$, we acquire the following interpolation inequality
 \begin{equation*}
\||\nabla|^{1/2}u\|_{L^2}\leq \|(-\Delta)^{\alpha/2}u\|^{1/2\alpha}_{L^2}\|u\|^{1-1/2\alpha}_{L^2},
\end{equation*}
On the other hand, using Lemma \ref{lem10} and Lemma \ref{lem11}, we get
\begin{equation*}
\begin{aligned}
\mathbb{E}\left[\left|\mathcal{M}_{\varphi_{R}}[u(t)]\right|\right]&\leq C(M(u_0), \varphi_{R}) \mathbb{E}\left(\||\nabla|^{1/2}u\|^2_{L^2}+\||\nabla|^{1/2}u\|_{L^2}\right)\\
&\leq  C(M(u_0), \varphi_{R})\mathbb{E}\left[\left(\|(-\Delta)^{\alpha/2}u(t)\|^{1/\alpha}_{L^2}
+\|(-\Delta)^{\alpha/2}u(t)\|^{1/2\alpha}_{L^2}\right)\right]\\
&\leq C(M(u_0), \varphi_{R})\mathbb{E}\left[\|(-\Delta)^{\alpha/2}u(t)\|^{1/\alpha}_{L^2}\right].
\end{aligned}
\end{equation*}
Therefore for all $t\geq t_1$, we conclude from \eqref{M} that
\begin{equation}\label{inequality}
\mathbb{E}\left[\mathcal{M}_{\varphi_{R}}[u(t)]\right]\leq -A\int^{t}_{t_1}\mathbb{E}\left[\left|\mathcal{M}_{R}[u(\tau)]\right|^{2\alpha}\right]d\tau,
\end{equation}
where $A:=C(M(u_0), \varphi_R)>0$.

Step 3: Define $z(t)=\int^{t}_{t_1}\mathbb{E}\left[\left|\mathcal{M}_{\varphi_{R}}[u(\tau)]\right|^{2\alpha}\right]d\tau$. Clearly, the function $z(t)$ is strictly increasing and nonnegative. Moreover,  by Jensen's inequality and \eqref{inequality}, we have
\begin{equation*}
z^{\prime}(t)=\mathbb{E}\left[\left|\mathcal{M}_{\varphi_{R}}[u(t)]\right|^{2\alpha}\right] \geq \left|\mathbb{E}[\mathcal{M}_{\varphi_{R}}[u(t)]] \right|^{2\alpha} \geq A^{2\alpha}z(t)^{2\alpha}.
\end{equation*}
Hence, if we integrate this differential inequality on $[t_1, t]$, we obtain
\begin{equation*}
z(t)\geq \frac{z(t_1)}{[1-(2\alpha-1)A^{2\alpha}z(t_1)^{2\alpha-1}(t-t_1)]^{\frac{1}{2\alpha-1}}}.
\end{equation*}
Then, we conclude that
\begin{equation*}
\mathbb{E}\left[\mathcal{M}_{R}[u(t)]\right]\leq -Az(t)\leq \frac{-Az(t_1)}{[1-(2\alpha-1)A^{2\alpha}z(t_1)^{2\alpha-1}(t-t_1)]^{\frac{1}{2\alpha-1}}}.
\end{equation*}
Since $2\alpha>1$, this inequality implies that $\mathbb{E}\left[\mathcal{M}_{R}[u(t)]\right]\rightarrow -\infty$ as $t\uparrow t_{*}$ for some finite time $t_*=t_1+\left[(2\alpha-1)A^{2\alpha}z(t_1)^{2\alpha-1}\right]^{-1} <+\infty$. Therefore, the solution $u(t)$ cannot exist for all times $t\geq 0$.
This ends the proof.

\end{proof}

\noindent
{\bf Acknowledgments} The research of A. Zhang was partially supported by the Natural Science Foundation of Changsha City of China (Grant No. kq2402194). The research of Y. Zhang was partially supported by the Natural Science Foundation of Henan Province of China (Grant No. 232300420110).

\noindent
{\bf Author Contributions} Ao Zhang and Yanjie Zhang wrote the main manuscript text. All authors reviewed the manuscript.

\noindent
{\bf Data Availability}  No datasets were generated or analyzed during the current study.

\section*{Declarations}
\noindent
 {\bf Conflict of Interest} The authors declare that they have no conflict of interest.

\end{document}